\documentclass [11pt]{amsart}

\usepackage{amsthm, amsmath, amssymb,latexsym}
\usepackage[latin1]{inputenc}
\usepackage{ amsthm}

\usepackage{accents}
\usepackage[all]{xy}
\usepackage{pb-diagram,pb-xy}
\usepackage{stmaryrd}

\usepackage{latexsym}

\newtheorem{theorem}{Theorem}[section]
\newtheorem{lemma}[theorem]{Lemma}
\newtheorem{corollary}[theorem]{Corollary}
\newtheorem{proposition}[theorem]{Proposition}
\newtheorem{remark}[theorem]{Remark}
\newtheorem{definition}[theorem]{Definition}

\newcommand{\nc}{\newcommand} 
\nc{\cH}{{\mathcal H}}
\nc{\cA}{{\mathcal A}}
\nc{\cG}{{\mathcal G}}
\nc{\cC}{{\mathcal C}}
\nc{\cO}{{\mathcal O}}
\nc{\cI}{{\mathcal I}}
\nc{\cB}{{\mathcal B}}
\nc{\cY}{{\mathcal Y}}
\nc{\cK}{{\mathcal K}} 
\nc{\cX}{{\mathcal X}}
\nc{\cS}{{\mathcal S}}
\nc{\cE}{{\mathcal E}}
\nc{\cF}{{\mathcal F}}
\nc{\cZ}{{\mathcal Z}}
\nc{\cQ}{{\mathcal Q}}
\nc{\cN}{{\mathcal N}}
\nc{\cP}{{\mathcal P}}
\nc{\cL}{{\mathcal L}}
\nc{\cM}{{\mathcal M}}
\nc{\cT}{{\mathcal T}}
\nc{\cW}{{\mathcal W}}
\nc{\cU}{{\mathcal U}}
\nc{\cJ}{{\mathcal J}}
\nc{\cV}{{\mathcal V}}
\nc{\bH}{{\mathbb H}}
\nc{\bA}{{\mathbb A}}
\nc{\bG}{{\mathbb G}}
\nc{\bC}{{\mathbb C}}
\nc{\bO}{{\mathbb O}}
\nc{\bI}{{\mathbb I}}
\nc{\bB}{{\mathbb B}}
\nc{\bY}{{\mathbb Y}}
\nc{\bK}{{\mathbb K}} 
\nc{\bX}{{\mathbb X}}
\nc{\bS}{{\mathbb S}}
\nc{\bE}{{\mathbb E}}
\nc{\bF}{{\mathbb F}}
\nc{\bZ}{{\mathbb Z}}
\nc{\bQ}{{\mathbb Q}}
\nc{\bN}{{\mathbb N}}
\nc{\bP}{{\mathbb P}}
\nc{\bL}{{\mathbb L}}
\nc{\bM}{{\mathbb M}}
\nc{\bT}{{\mathbb T}}
\nc{\bW}{{\mathbb W}}
\nc{\bU}{{\mathbb U}}
\nc{\bD}{{\mathbb D}}
\nc{\bJ}{{\mathbb J}}
\nc{\bV}{{\mathbb V}}
\nc{\bbZ}{{\mathbb Z}}
\nc{\bR}{{\mathbb R}}
\nc{\fr}{{\rightarrow}}
\nc{\co}{{\nabla}}
\newcommand{\la}{\longrightarrow}
\nc{\cu}{{\barline{\nabla}}}

\DeclareMathOperator{\Pic}{Pic}
\DeclareMathOperator{\Alb}{Alb}
\DeclareMathOperator{\Aut}{Aut}
\DeclareMathOperator{\End}{End}
\DeclareMathOperator{\Hilb}{Hilb}
\DeclareMathOperator{\Spec}{Spec}
\DeclareMathOperator{\Sym}{Sym}

\begin{document}

\title {Dihedral Monodromy and Xiao Fibrations} 
                               %
                               %
\author{Alberto Albano}
\address{Alberto Albano  \\ Universit\`a degli Studi di Torino  \\ Dipartimento di Matematica \\ Via Carlo ALberto 10  \\ 10123 Torino, Italy  }
\email{alberto.albano@unito.it}

\author{Gian Pietro Pirola}
\address{Gian Pietro Pirola  \\ Universit\`a degli Studi di Pavia  \\ Dipartimento di Matematica \\ Via Ferrata, 1  \\ 27100 Pavia, Italy  }
\email{gianpietro.pirola@unipv.it}

\thanks{The second author was partially supported by INdAM (GNSAGA); PRIN 2012 \emph{``Moduli, strutture geometriche e loro applicazioni''}; FAR 2013 (PV) \emph{``Variet\`a algebriche, calcolo algebrico, grafi orientati e topologici''}}

\keywords{Fibrations, irregular surfaces, Prym varieties}
\subjclass[2010]{14D06, 14J29, 14H40}

\begin{abstract} 
We construct three new families of fibrations $\pi : S \to B$ where $S$ is an algebraic complex surface and $B$ a curve that violate Xiao's conjecture relating the relative irregularity and the genus of the general fiber. The fibers of~$\pi$ are certain \'etale cyclic covers of hyperelliptic curves that give coverings of~$\bP^1$ with dihedral monodromy.

As an application, we also show the existence of big and nef effective divisors in the Brill-Noether range.
\end{abstract}

\maketitle

\section{Introduction}
\label{s:intro}
Let $S$ be a smooth complex projective surface, $B$ a smooth curve of genus $b$ and $\pi : S \to B$ a fibration, i.e., a surjective morphism with connected fibers. Let $C$ be the general fiber of $\pi$ and $g_C$ its genus.
 Let $q = \dim H^1(S, \cO_S)$ be the irregularity of $S$ and   $q_\pi = q - b$ the \emph{relative irregularity} of the fibration. 
 The fibration is called \emph{isotrivial} if the smooth fibers are all isomorphic. 
 
 Assume that the fibration is not isotrivial and $b=0,$ that is $B = \mathbb P^1$  is the projective line. Under these hypotheses Xiao proved in \cite{Xi}  the inequality
\[ 
q\leq \frac {g_C + 1}{2}
\] 
and he conjectured that the inequality   
\begin{equation}
q_\pi \leq  \frac {g_C + 1}{2}
\label{xiao}
\end{equation} 
holds in general for non isotrivial fibrations, see also \cite{Xi1}, \cite{moe}.

It was shown in \cite{P}  that this conjecture is false by constructing a fibration~$\pi$ with $g_C = 4$ and $q_\pi = 3$ and the failure of the conjecture was linked to the non triviality of a certain higher Abel-Jacobi map.

This motivates the following
\begin{definition} Let $S$ be a surface. A fibration $\pi: S \to B$ with general fiber~$C$ is called a \emph{Xiao fibration}
if 
\[
q_\pi > \frac{g_C + 1}{2}.
\]
\end{definition}

Up to now the only examples of Xiao fibrations were the ones constructed in~\cite{P}. In this paper we construct three new families of Xiao fibrations associated to cyclic \'etale covers of hyperelliptic curves.  

Let us explain the main idea: let $E$ be an hyperelliptic curve of genus~$g$ and $f : C \to E$ a cyclic \'etale cover of odd prime order~$p$. In this situation the hyperelliptic involution lifts to an automorphism of~$C$ and let $D$ be the quotient of~$C$ by this automorphism (see \cite{Ac}, \cite{Far}). The lift of the involution and the deck transformations of the \'etale cover generate a dihedral group~$D_p$ of automorphisms of $C$. This group is also the monodromy of the induced ramified cover $D \to \mathbb{P}^1$. 

Let $P(C, E)$ be the (generalized) Prym variety associated to~$f$. In~\cite{R} it is proved that
$P(C,E)$ is the product of the jacobian~$J(D)$ with itself and hence $J(C)$ is isogenous to $J(D) \times J(D) \times J(E)$.
We have $g_C = p(g - 1) + 1$ and $g_D = (p - 1)(g - 1)/2$.

This construction gives a map
\begin{equation}
\psi : \cH_{g,p} \to \cM_{(p - 1)(g - 1)/2}
\end{equation}
 from the moduli space $\cH_{g,p}$ of unramified cyclic covers of degree~$p$ of hyperelliptic curves of genus~$g$ to the moduli space $\cM_{g_D}$ of curves of genus $g_D$.
 
We study the fibers of $\psi$ and determine when they are positive dimensional (Proposition~$\ref{finitefibers}$). In those cases, an irreducible component of the fiber gives a family of curves~$C$ whose Jacobians have a fixed part~$J(D) \times J(D)$. 

In general, from a family of curves one can construct fibrations that have the curves~$C$ of the family as fibers. In our situation, the geometry of the \'etale covers allows us to construct these fibrations as subvarieties of an appropriate Hilbert scheme of the surface~$D\times D$
and hence we have a lower bound on the relative irregularity. We will prove the
\begin{theorem}
\label{maintheorem}
There exist Xiao fibrations $\pi : S \to B$ with general fiber $C$ in the following cases:
\begin{enumerate}
\item $E$ of genus~$g = 2$ and covers of degree~$p = 5$. This gives $g_C = 6$, $q_\pi = 4$; 
\item $E$ of genus~$g = 4$ and covers of degree~$p = 3$. This gives $g_C = 10$, $q_\pi = 6$; 
\item $E$ of genus~$g = 3$ and covers of degree~$p = 3$. This gives $g_C = 6$, $q_\pi = 4$;
\end{enumerate}
\end{theorem}
First of all in section~$\ref{Xiaofibr}$ we construct the fibrations associated to the positive dimensional fibers of~$\psi$ and then
we analyze the three cases respectively in sections~$\ref{g2p5}$, $\ref{g4p3}$ and~$\ref{g3p3}$. 

In case~$1$ we compute the differential of the Prym map to show that the fibers of~$\psi$ have dimension~$1$ and hence the fibrations are in fact surfaces.

In case~$2$ we find that an irreducible component of the fiber of~$\psi$ is the curve $D$ itself. This allows us to construct the surface~$S$ as a ramified double cover of~$D \times D$. From this explicit description we can compute all the invariants of the surface~$S$ (Theorem~$\ref{invariants})$.

In case~$3$ let $F$ be the fiber of~$\psi$. Then $F$ has dimension~$2$ and generically the genus of $C$ is~$7$ so we obtain a threefold such that for a general curve $X$ inside $F$ the corresponding fibered surface is not a Xiao fibration. We then compactify the fiber of~$\psi$ and analyze the singular curves we obtain at the limit. One can normalize these curves obtaining a surface with the same relative irregularity and geometric genus of the generic fiber equal to~$6$, giving again a Xiao fibration.

\medskip
We note that for the Xiao fibration found in~\cite{P} as well as for all these new ones, the Xiao conjecture fails only by $1/2$, i.e., in all cases one has
\begin{equation}
\label{xiaoline}
q_\pi = \left\lceil \dfrac{g_C + 1}{2} \right\rceil 
\end{equation}

This bound has appeared recently in the work of Barja, Gonz\'alez-Alonso and Naranjo (see \cite{V} and \cite{BGN}). The main result of~\cite{BGN} says that for a non isotrivial fibration $\pi : S \to B$ one has
\begin{equation*}
q_\pi \le g_C - c_\pi
\end{equation*}
where $c_\pi$ is the Clifford index of the general fiber. When $c_\pi = \left\lfloor \dfrac{g_C - 1}{2} \right\rfloor$, i.e., $c_\pi$ is equal to the Clifford index of the general curve of genus~$g$, the previous inequality becomes
\begin{equation}
\label{xiaomod}
q_\pi \le \left\lceil \dfrac{g + 1}{2} \right\rceil 
\end{equation} 
In~\cite{BGN} it is conjectured that the inequality~$(\ref{xiaomod})$ holds for all non isotrivial fibrations. Our work seems to confirm this conjecture. It is an interesting problem to provide examples of Xiao fibrations with $g_C$ arbitrarily high.

\medskip
Our examples in cases~$1.$ and~$2.$ provide also an answer to a question posed in~\cite{mpp} (Question~$8.6$). In fact we have
\begin{proposition}
There exist surfaces $S$ and nef and big effective divisors $C$ on $S$ in the Brill-Noether range, i.e., such that $q(S) < g_a(C) < 2q(S) - 1$, where $g_a(C)$ is the arithmetic genus of~$C$.
\end{proposition}

\begin{proof}
In cases~$1.$ and~$2.$ the curves~$C$ embed into $S = D \times D$ with positive self-intersection by Lemma~$\ref{etale}$ and~$\ref{C2}$. Since $q(S) = 2 g(D)$ we have $q(S) = (p-1)(g-1)$ and since $C$ is smooth we have $g_a(C) = g_C = p(g-1) +1$, and so $C$ is in the Brill-Noether range.
\end{proof}

In case~$2.$ the divisor~$C$ is even ample (see Remark~$\ref{sym2}$).

\smallskip
\textbf{Acknowledgements.} The authors thank Rita Pardini for catching an error in the first draft of this paper.

\section{Dihedral groups and hyperelliptic curves}

We recall the following well known result (see \cite{Ac}, \cite{Far}, \cite{R}):

\begin{proposition} Let  $C \to E$ be an \'etale abelian Galois cover, where $E$ is an hyperelliptic curve. Then the hyperelliptic involution lifts to an involution on~$C$. If morever the Galois group is cyclic, then the group generated by the Galois group and a lift of the involution is a dihedral group~$D_n$ of order~$2n$, where $n$~is the order of the Galois group.
\end{proposition}

In the cyclic case we consider the following commutative diagram:
\begin{equation}
\label{eq:dihedral}
\xymatrix{
C \ar[r]^{p:1}_f \ar[d]_{2:1}^{\rho} & E \ar[d]^{2:1} \\
D \ar[r]_{p:1} & \bP^1 \\
} 
\end{equation}
where $E$ is an hyperelliptic curve of genus~$g$, $E\to\bP^1$ is the hyperelliptic quotient and $f : C \to E$ is \'etale and abelian with cyclic Galois group $H$ of odd order~$p$. By the previous Proposition, $C \to \bP^1$ is Galois with Galois group $G = D_p$. 

Then $\rho: C\to D$ is the quotient by a lift of the hyperelliptic involution and $D\to \bP^1$ is a non-Galois ramified cover with dihedral monodromy.

This can be realized as follows: fix an hyperelliptic curve~$E$ of genus~$g$ and a cyclic subgroup $H'$ of order~$p$ of $\Pic^0(E)$. This gives the cover $C\to E$. Note that $g(C) = g_C = p(g-1) + 1$.

Now let $j \in D_p \subseteq \Aut(C)$ be a lift of the hyperelliptic involution. Then $j$ has $2g + 2$ fixed points, one over each Weierstrass point of~$E$ (which are the ramification points of the double cover $E\to \bP^1$). We use here the fact that the order~$p$ is odd.

Hence the genus of $D$ is $g_D = (p-1)(g-1)/2$. 

The ramification of the cover $D\to \bP^1$ is: over every branch point of the hyperelliptic covering there are $1 + (p-1)/2$ points. One of these points is non-ramified, the others have ramification index~$1$.

Conversely, starting with $D\to \bP^1$ with the above ramification and dihedral monodromy, its Galois closure is $C\to D \to \bP^1$.

Associated to the \'etale cover $f : C \to E$ there is a Prym variety $P(C, E)$ defined as the connected component of the identity of the kernel 
of the map $f_* : J(C) \to J(E)$ and $J(C)$ is isogenous to the product $P(C, E) \times J(E)$. The main theorem of \cite{R} identifies precisely the Prym variety:

\begin{theorem}[\cite{R}, Theorem 1]
\label{Ries}
There is an isomorphism of abelian varieties
\[
P(C, E) \cong J(D) \times J(D).
\]
Morever, if $h$ is a generator of the cyclic subgroup $H \subseteq \Aut(C)$, then the endomorphism
$\eta = h^* + (h^{-1})^*$ of $J(C)$ induces a nontrivial automorphism of $J(D)$ for $p > 3$. 
\end{theorem} 

\begin{corollary} \ 
\label{isog}
\begin{enumerate}
\item $J(C)$ is isogenous to $J(D) \times J(D) \times J(E)$.
\item For $p > 3$ the curve $D$ is special in moduli since its Jacobian has non trivial automorphisms.
\end{enumerate}
\end{corollary}

Hence $\End(J(D))\otimes \bQ$ contains at least $\bQ(\eta)$ which is isomorphic to the maximal real subfield 
of the cyclotomic field $\bQ(\zeta)$ with $\zeta^p = 1$. For more results on endomorphism of Jacobians see \cite{E}.

\medskip
Let $\cH_{g,p}$ be the moduli space of unramified cyclic covers of degree~$p$ of hyperelliptic curves of genus~$g$. A point in $\cH_{g,p}$ is (up to isomorphism) a pair $(E, H')$ where $E$ is an hyperelliptic curve of genus~$g$ and $H'$ is a cyclic subgroup of order~$p$ of $\Pic^0(E)$. The dihedral construction of  diagram~$(\ref{eq:dihedral})$ determines uniquely the isomorphism class of $D$, since any two lifts of the hyperelliptic involution are conjugated in~$\Aut(C)$ and hence gives a morphism
\begin{equation}
\label{psi}
\psi : \cH_{g,p} \to \cM_{(p - 1)(g - 1)/2}
\end{equation}
from the moduli space $\cH_{g,p}$  to the moduli space $\cM_{g_D}$ of curves of genus $g_D = (p - 1)(g - 1)/2$.

 The image of $\psi$ is clearly contained in the locus of $p$-gonal curves. When $p = 3$  the closure of the image is the trigonal locus since for $D \to \bP^1$ a map of degree~$3$ with simple ramifications, the monodromy is the full symmetric group~$\cS_3 = D_3$ and hence $D$ is in the image of the map~$\psi$. These curves form an open subset of the trigonal locus which is irreducible.

 We study now the fibers of this map and for this we analyze the correspondence associated to the endomorphism~$\eta$ of $J(C)$.

Recall that $h \in H \subseteq D_n$ is a generator of the cyclic subgroup $H$ and $j$ is a lift of the hyperelliptic
involution. Let $j_1 = hj$ and note that $j_1$ is again an involution. Let 
$\gamma: C \to D\times D$ be defined by 
\[ 
\gamma(x) = (\rho_j(x),\rho_{j_1}(x)).
\]  
where $\rho_j$, $\rho_{j_1}$ are the quotient maps associated to the involutions. Note that $\rho_{j_1} = \rho_j \circ h^{(p-1)/2}$.

We have:
\begin{lemma}
\label{etale}  
The map $\gamma$ is an embedding.
\end{lemma}

 \begin{proof}
 If $x$ is not a fixed point for $j$ it follows that the map $\rho_j(x)$ is smooth at $x,$ that is the differential $d\rho_j$ is injective at $x,$ a fortiori $d\gamma(x)$ is injective. Therefore $d\gamma(x) $ can fail to be injective only if $j(x) = j_1(x) = x$ and this implies $h(x) = x$. But since $f$ is \'etale, $h$ does not have fixed points.
 
 In a similar way we see that $\gamma$ is injective. Assume  by contradiction that $\gamma(x)=\gamma(x'),$ but $x\neq x'$. Then $j(x) = j_1(x) = x'$ and $h(x') = hj(x) = j_1(x) = x'$ and as before $h$ would have a fixed point.
  \end{proof}

\begin{remark}
\label{r:surj}
 The proof of Theorem~$\ref{Ries}$ shows that the induced map $\gamma_* : J(C) \to J(D)\times J(D)$ is surjective. We will need this remark in Lemma~$\ref{qrel}$.
\end{remark}  

Let $(E, H')\in \cH_{g,p}$, let $[D] = \psi(E,H')$ and let $X$ be an irreducible component of the fiber~$\psi^{-1}([D])$. The discussion above shows that there is a morphism 
\begin{equation}
\label{HilbSch}
\alpha : X \to \Hilb(D \times D)
\end{equation}
from $X$ to a suitable Hilbert scheme of $D \times D$ given as follows: to a point $(E, H') \in X$ we associate the subscheme $\gamma(C)$ of $D \times D$.

We now compute the self-intersection of $\gamma(C)$ inside the surface $D \times D$. 
\begin{lemma}
\label{C2} 
\[
\gamma(C)^2 = 8 - 2(g - 1)(p - 2)
\]
\end{lemma}

\begin{proof}
Use the genus formula for $\gamma(C)$ inside $D \times D$, the formula for the canonical bundle of the product of two curves and the fact that $\gamma(C) \cdot D' = 2$, where $D' = D \times \{P\}$ since the degree of the map $\rho_j : C \to D$ is $2$.
\end{proof}

A similar computation appears in \cite{E}, Proposition~$4.1$ where the self intersection is expressed in terms of characters 
of the dihedral group.

\begin{proposition} 
\label{finitefibers}
The map $\psi$ has finite fibers if and only if $p \ge 7$, $p = 5$ and $g \ge 3$ or $p = 3$ and $g \ge 5$
\end{proposition}

\begin{proof}
If the map $\psi$ has positive dimensional fibers, then the image of $C$ inside $D \times D$ must move in an algebraic family. This implies $\gamma(C)^2 \ge 0$ and so we obtain all the cases in the statement except for $p = 3$ and $g = 5$.

In this case the curve $D$ is a trigonal curve of genus $4$ and $C$ is the graph in $D \times D$ of the trigonal correspondence. 
Since $D$ has only one or two $g^1_3$, the fiber is finite also in this case.

We show now that for $p$ and $g$ not in the given ranges the fibers are positive dimensional. Note that $\cH_{g,p}$ and $\cM_{(p - 1)(g - 1)/2}$ are irreducible.

When $p = 3$ and $g \le 4$ we have $\dim \cH_{g,3} > \dim \cM_{g-1}$ and so the fibers are positive dimensional.

The last case is $p=5$ and $g=2$. In this case $D$ has also genus~$2$ so $\dim \cH_{2,5} = \dim \cM_2 = 3$. Since $J(D)$ has non trivial endomorphisms by Theorem~\ref{Ries} the curve  $D$ is not a general curve of genus~$2$ and so the image of $\psi$ has dimension at most~$2$.
\end{proof}

\begin{remark}
\label{psiPrym}
Let $\cP : \cH_{g, p} \to \cA'_{(p-1)(g-1)}$ be the Prym map that to~$(E, H')$ associates $(P, \theta_P)$ where $P = P(C, E)$ is the Prym variety of the cover $C \to E$ determined by $H'$ and $\theta_P$ is the natural polarization induced by $J(C)$. Note that $\theta_P$ in general is not principal (see~\cite{R} for details). On the other hand, composing the map $\psi$ with the Torelli map~$t$ we obtain a map $T : \cH_{g, p} \to \cA_{(p-1)(g-1)}$ given by
$T(E, H') = J(D) \times J(D)$ with the product polarization.

By Theorem~$\ref{Ries}$ the abelian varieties $P$ and $J(D)\times J(D)$ are isomorphic. Since the Torelli map is injective and an abelian variety has at most a countable number of polarizations
 the fibers of $\psi$ have the same dimension as the fibers of the Prym map~$\cP$.
\end{remark}

We close this section noting that the above construction and Theorem~\ref{Ries} give families  of curves
of genus $(p-1)(g-1)/2$ whose Jacobians have a non trivial algebra of endomorphisms. When the fibers of $\psi$ are finite, these families have dimension $2g - 1$. We note that setting $g = 2$ and $p \ge 7$ we recover (at least in characteristic 0) part $(1)$ of the Main Theorem of \cite{E}. When $g=2$ and $p=5$ the family has dimension~$2$.

\section{Xiao fibrations}
\label{Xiaofibr}

Any subvariety of $\cM_g$ gives rise to some fibration whose general fibers are the genus $g$ curves belonging to the family (see~\cite{K} for a precise statement).
We consider here the subvarieties given by the positive dimensional fibers of the maps $\psi$ defined in~$(\ref{psi})$.
In this case the corresponding fibrations can be more easily constructed by using the universal family of appropriate Hilbert schemes. 

For $X$ an irreducible component of a fiber of $\psi$ we consider the morphism $\alpha : X \to \Hilb(D \times D)$ given above in~$(\ref{HilbSch})$. Let $Y$ be the irreducible component of~$\Hilb(D\times D)$ containing the image $\alpha(X)$ and, if necessary, consider its reduced structure. Let $\cC$ be the universal family over~$Y$.  Let $\overline{X}$ be a smooth completion of $X$. As the Hilbert scheme is projective, the morphism $\alpha$ extends to a rational map $\alpha : \overline{X} \dashrightarrow Y$ and after  blowing up, if necessary, we get a morphism $\alpha : B \to Y$. The pullback of the universal family over~$Y$ gives a fibration 
\begin{equation}
\label{e:fibration}
\pi : S_D \to B
\end{equation}
whose general fibers are curves~$C$ that are cyclic covers of the curves~$E$ in the fiber of $\psi$ over $[D]$
  of genus $g_C = p(g - 1) + 1$. 

\begin{lemma} 
\label{qrel}
For the general $D$ in the image of $\psi$ the relative irregularity of $S_D$ is $2g_D = (p-1)(g-1)$.
\end{lemma}

\begin{proof} Let $\pi_* : \Alb(S) \to J(B)$ be the map from the Albanese variety of~$S$ to~$J(B)$ induced by $\pi$ and let $K$ be the connected component of the identity of the kernel of~$\pi_*$. By definition, the relative irregularity~$q_\pi$  is the dimension of~$K$.

Let $C_t = \pi^{-1}(t)$ for $t \in B$ and let $E_t$ be the corresponding hyperelliptic curve. Since the family $\{E_t\}$ is not constant in moduli, also $\{C_t\}$ has varying moduli. 

The composition $J(C_t) \to \Alb(S) \to J(B)$ is trivial since $C_t$ is a fiber of $\pi$ and hence the image of $J(C_t)$ is contained in $K$ and as in \cite{P}, (0.5), one has that the image of~$J(C_t)$ is in fact equal to~$K$.

The embeddings $\gamma_t: C_t \to D\times D$ induce a map $S \to D\times D$ which is surjective since the curves $C_t$ do not have not constant moduli and hence a surjective map $\Alb(S) \to \Alb(D\times D) = J(D)\times J(D)$. Moreover $(\gamma_t)_*$ factors through $\Alb(S)$. By Remark~\ref{r:surj} the map $(\gamma_t)_*$ is surjective and hence the restriction to $K \to J(D)\times J(D)$ is surjective. This shows $q_\pi \ge 2g_D$.

Recall now that $J(C_t)$ is isogenous to $J(D)\times J(D) \times J(E_t)$ (Corollary~$\ref{isog}$) and so there is a surjective map $J(D)\times J(D) \times J(E_t) \to K$. The image of $J(E_t)$ is constant in $K$. If at least one curve $E_t$ in the family has indecomposable Jacobian, then this image is ~$0$ and so the relative irregularity is exactly $2g_D$. 
\end{proof}

By Proposition~$\ref{finitefibers}$ there are $4$~cases in which we obtain a positive dimensional $B$. When $B$ is a curve, the fibration $S_D$ is a surface and we may ask if it is a Xiao fibration. This cannot happen for $p=3$, $g=2$ but we will see that in the other three cases we obtain Xiao fibrations. We will study these cases separately.

\section{The case $g = 2$, $p = 5$}
\label{g2p5}
Our first task is to show that the fibers of $\psi : \cH_{2,5} \to \cM_2$ have dimension~$1$. By Remark~$\ref{psiPrym}$ it is enough to compute the dimension of the fibers of the Prym map.

Let $(E, H') \in \cH_{2,5}$ and $f: C \to E$ the associated  \'etale covering. For  $L$ a generator of $H'$ we have that $ C = \Spec \left(\bigoplus_{i=0}^4 \cO_E(L^i) \right)$ and  $C$ is a genus $6$ curve. Let $K_C$, $K_E$ be canonical bundles of $C$ and $E$ respectively. The Chevalley-Weil relations are (see e.g. \cite{N}, \cite{W}):
\[f_\ast K_C =  \bigoplus_{i=0}^4 K_E\otimes L^i
\] 
\[H^0(C,K_C) = \bigoplus_{i=0}^4 H^0(E, K_E\otimes L^i)
\]
Let $P = P(C, E)$ be the Prym variety: it is an abelian variety of dimension four and $\Omega^1_P \cong \left(\bigoplus_{i=1}^4 H^0(E, K_E\otimes L^i) \right) \otimes \cO_P$ since these are the $1$-forms not invariant under the action of the covering group. Hence
\[
H^0(P, \Omega^1_P) \cong \bigoplus_{i=1}^4 H^0(E, K_E\otimes L^i)
\]
Under the inclusion $P\subset J(C),$ the principal polarization of $J(C)$ defines a polarization~$\theta_P$ on $P$. As in Remark~\ref{psiPrym}, sending $(E, H')$ to $(P, \theta_P)$ gives the Prym map $\cP : \cH_{2, 5} \to \cA'_4$.

We have an inclusion $H' \cong \mathbb Z/5\mathbb Z$ in $\Aut P$, the automorphism group of the polarized variety $(P,\theta_P).$  Clearly  the image of $\cP$ is contained in the locus~$\cA'_4(5)$ of abelian fourfolds with $\mathbb Z/5\mathbb Z$ automorphisms.  The Zariski cotangent space to $\cA'_4(5)$ is isomorphic to the invariant subspace $\Sym^2 H^0(P,\Omega^1_P)^{H'}$ of $\Sym^2 H^0(P,\Omega^1_P)$.

The codifferential of~$\cP$ can be seen as a linear map 
\[
d\cP^* : \Sym^2 H^0(P,\Omega^1_P)^{H'} \to H^0(E, 2K_E).
\]
since $H^0(E, 2K_E)$ is isomorphic to the cotangent space  of~$\cH_{2,5}$. We have that
\[
\Sym^2 H^0(P,\Omega^1_P)^{H'} \cong \big[ H^0(K_E\otimes L) \otimes H^0(K_E\otimes L^4) \big] \oplus \big[H^0(K_E\otimes L^2) \otimes H^0(K_E\otimes L^3)\big]
\]
and $d\cP^\ast $ can be identified with the map $\mu$ induced by multiplication.

\begin{lemma} 
The map $\mu$ is injective.
\end{lemma}

\begin{proof} Since $h^0(K_E\otimes L) = 1$ we can write $K_E\otimes L = \cO_E(P+Q)$. The hyperelliptic involution~$\iota$ on~$E$ acts as $-1$ on $J(E)$ hence we have $\cO_E(\iota(P) + \iota(Q)) = K_E\otimes L^{-1} = K_E\otimes L^4$, since the canonical bundle is invariant under automorphisms. Suppose that $\mu$ is not injective. 

We then get an equation  :
$\omega_1\cdot\omega_4+\omega_2\cdot\omega_3=0\in H^0(E,2K_E), $ where $\omega_i$ are suitable generators of $H^0(E,K_E \otimes L^i)$. This gives a relation among the divisors: 
\[
P + Q + \iota(P) + \iota(Q) = (\omega_1) + (\omega_4) = (\omega_2) + (\omega_3).
\] 
 We can then assume $\cO_E(P + \iota(Q)) = K_E \otimes L^2$ and $\cO_E(\iota(P) + Q) = K_E\otimes L^3.$ 
It then follows  $L=\cO_E(\iota(Q) - Q)$ and since $K_E = \cO_E(Q + \iota(Q))$, we have 
\[
K_E\otimes L = \cO_E(2\iota(Q))
\]
and since $h^0(K_E\otimes L) = 1$ it must be $P = Q = \iota(Q)$. But this would give $L = \cO_E$ which is  a contradiction.
\end{proof}

\begin{proposition}
The map $\psi : \cH_{2, 5} \to \cM_2$ has fibers of dimension~$1$.
\end{proposition}

\begin{proof}
Look at the Prym map
\[
\cP : \cH_{2, 5} \to \cA'_4(5) \subseteq \cA'_4
\]
The codifferential is injective and so the differential is surjective. Hence the dimension of the image is~$\dim\cA'_4(5) = 2$ and so the fibers have dimension~$1$. 

By Remark~$\ref{psiPrym}$ the fibers of $\psi$ have also dimension~$1$.
\end{proof}

Let $D \in\psi(\cH_{2,5})$ a generic curve and let $X$ be an irreducible component of the fiber $\psi^{-1}(D)$. By the general construction explained in Section~$\ref{Xiaofibr}$ we obtain a surface $S_D$ with a fibration
\[
\pi : S_D \to B
\]
By what we have seen, we get

\begin{proposition}
The fibration $\pi : S_D \to B$ is a Xiao fibration with relative irregularity $q_\pi = 4$ and genus of the general fiber $g_C = 6$. 
\end{proposition}

This is case~$1$ of Theorem~$\ref{maintheorem}$.

\section {The case $g = 4$, $p = 3$}
\label{g4p3}
We present here an explicit example of a Xiao fibration. Let $\psi : \cH_{4, 3} \to \cM_3$.
In this case we can identify an irreducible component of the fiber $\psi^{-1}(D)$ as being simply the curve $D$ itself.

In fact, let $D$ be a smooth plane quartic, i.e., a non hyperelliptic curve of genus~$3$. A point $P\in D$ gives a $g^1_3$ obtained as $|K_D - P|$. Let $f_P : D \to \bP^1$ be the map given by this linear series and assume that the map~$f_P$ has simple ramification points, i.e., the point $P$ is not on any flex tangent to~$D$. Then the monodromy of $f_P$ is the symmetric group~$\cS_3$ and let $C_P \to D \to \bP^1$ be the Galois closure. Let $E_P$ be the quotient of $C_P$ by the alternating group~$\cA_3 = \bZ_3$. Then $C_P$ is a curve of genus $10$, $E_P$ is an hyperelliptic curve of genus~$4$ and the cover $C_P \to E_P$ is \'etale and hence gives a point in the fiber~$\psi^{-1}(D)$. 

Since all $g^1_3$ on $D$ are of this kind, we find a copy of (an open subset of) $D$ inside the fiber~$\psi^{-1}(D)$. We now give a geometric construction of the Galois closure and of a smooth compactification~$S$ of the fibration. This will allow us to describe completely~$S$ and compute all of its numerical invariants.

\medskip    
Let $D\subset \bP^2$ be a smooth plane quartic curve as above and let $S\subset D \times D\times D$ be defined as
\[
S = \{(P,Q,R): \exists T\in D: P+Q+R+T \in |K_D|\}.
\]
Note that for $P$, $Q$, $R$ distinct the condition simply means that the three points are collinear. We consider the projections $\pi_i: S \la D,$ $i=1,2,3$ on the three factors. The map $\beta = (\pi_2, \pi_3) : S \la D\times D$ is a surjective $2:1$ map so that $S$ is a surface. In fact 
\[
\beta^{-1}((P, Q))= \{(R, P, Q),(T, P, Q)\}
\]
where $R$ and $T$ are the two other points of intersection of the line~$PQ$ with the curve~$D$. 

\begin{theorem} 
\label{explicit}
Set  $\pi=\pi_1,$ the first projection, $\pi: S\to D.$ Then $\pi$ is a
fibration with general fibers smooth of genus~$10$ and relative irregularity greater or equal than~$6$.
\end{theorem}

This is case~$2$ of Theorem~$\ref{maintheorem}$.

\begin{proof}
   To compute the genus of the fiber $C_P=\pi^{-1}(P)$ we let $k : C_P\to S$ be the inclusion. The restriction of $\beta$ gives a natural inclusion $\beta_P = \beta\circ k :  C_P\to D\times D$ and let $X_P = \beta(C_P)$ be the image. Since $C_P$ and $X_P$ are isomorphic, we  compute the arithmetic genus of~$X_P$. To do this, we determine the class of $X_P$ in $D\times D$ under numerical equivalence. 
       
Let $f_P: D\to \bP^1$ be the $3:1$ map obtained by projecting the plane curve~$D$ from the point $P$.
Since $C_P$ is given by triples $(P, Q, R) \in S$ with $P$~fixed, then $X_P$ is the closure of 
\[
\{(Q,R)\in D\times D \mid Q\neq R, f_P(Q) = f_P(R)\},
\]

Let $D_1 = \{P\} \times D$ and $D_2 = D \times \{P\}$ and $\Delta$ be the diagonal in~$D\times D$.
The self-intersection number $X_P^2$ can be computed by taking another point~$Q \in D$ and computing $X_P\cdot X_Q =\{(R,T), (T,R)\}$, where $R$ and $T$ are the two other points of intersection of the line~$PQ$ with the curve~$D$. Hence $X_P^2 = 2$. Moreover, $X_P\cdot D_1 = X_P\cdot D_2 = 2$ and by the Hurwitz formula $X_P\cdot \Delta = 10$.

Let now $H = 3(D_1 + D_2)-\Delta$. One has $H^2 = H\cdot D_1 = H\cdot D_2 = 2$ and hence $(H - X_P)\cdot (D_1 + D_2) = 0$. Since
\[
X_P \cdot H = X_P \cdot (3(D_1 + D_2)-\Delta) = 12 - 10 = 2
\]
we have 
\[
X^2_P = H^2 = X_P\cdot H = 2
\]
and hence
\[
(H - X_P)^2 = 0.
\]
Then by the Hodge index theorem $X_P$ is numerically equivalent to $H = 3(D_1 + D_2) - \Delta$ and using the adjunction formula $X_P$ has arithmetic genus~$10$.

We now show that $X_P$ is smooth unless the curve~$D$ has a flex $Q$ such that $|3Q + P| = K_D$. In fact, let $Q \in D$ be a simple ramification point for $f_P$. Choose a local coordinate~$z$ on $D$ centered at~$Q$ and a local coordinate~$w$ on $\bP^1$ centered at~$f_P(Q)$ such that in these coordinates the map~$f_P$ is given by $w = z^2$. Using the local coordinates on~$D\times D$ centered at~$(Q, Q)$ induced by~$z$, the points on the curve~$X_P$ different than $(Q,Q)$ are the pairs~$(x, y)$ such that~$x^2 = y^2$ and $x \ne y$. Then a local equation for~$X_P$ is~$x = - y$ which is smooth. If instead $|3Q + P| = K_D$, then there are similar coordinates systems such that locally the map is given by~$w = z^3$ and the above reasoning shows that a local equation for~$X_P$ is $x^2 + xy + y^2 = 0$, which is singular at the origin. 

Since $C_P$ is isomorphic to $X_P$ we obtain that the fibers of $\pi : S \to D$ are generically smooth of genus~$10$.

By Corollary~$\ref{isog}$ we know  that $J(C_P)$ has a fixed part of dimension~$6$ isogenous to~$J(D) \times J(D)$. Since $C_P$ is big and nef, we can prove that this fixed part is isomorphic to~$J(D) \times J(D)$ by showing that there is an inclusion
\[
J(D)\times J(D) = \Pic^0(D \times D) \hookrightarrow \Pic^0(C_P) = J(C_P).
\]
The proof is standard: Ramanujan vanishing gives an injection 
\[
H^1(D\times D, \cO_{D\times D}) \to H^1(C_P, \cO_{C_P})
\]
so if $L \in \Pic^0(D \times D)$ goes to zero in~$\Pic^0(C_P)$, then $L$ must be torsion. Then $L$ gives an unramified cyclic cover~$X$ of $D\times D$. Since $L$ is trivial on $C_P$, the pull-back of $C_P$ to $X$ splits in a number of connected components. Each component has positive self-intersection and they don't meet, and this contradicts the Hodge index theorem.

 Then the image of dual map $J(C_P)\to \Alb(S)$ has dimension $\geq 6.$
   It follows that the relative irregularity $q_\pi\geq 6$.    
\end{proof}   
 \bigskip

\begin{remark} 
\label{r:S_smooth}
A similar computation in local coordinates shows that the surface~$S$ is smooth if all the flexes are simple. When there are flexes of order four, the surface is singular. 
\end{remark}

\begin{remark} Let $\varphi : D\times D\times D \to D^{(3)}$ be the quotient map to the symmetric product. Then the image of the surface~$S$ is $D^1_3$, the set of divisors of degree~$3$ and $h^0 \ge 2$. $D^1_3$ is a ruled surface over~$D$ and the lines in the ruling are the~$g^1_3$ of~$D$.
\end{remark}

\begin{remark}
\label{sym2}
The curves $X_P$ can also be constructed in the following way: let $\varphi : D\times D \to D^{(2)}$ be the quotient map to the symmetric product and let 
\[
D_P = \{ P + Q \mid Q \in D\} \subseteq D^{(2)}
\]
($P$ is fixed). Let $\tau : D^{(2)} \to D^{(2)}$ be the canonical involution given by $\tau(P + Q) = R + T$ where $P + Q + R + T$ is a canonical divisor. Then
\[
X_P = \varphi^{-1} \left( \tau(D_P) \right)
\]
This also shows that $X_P$ is an ample divisor in $D\times D$.
\end{remark}

\begin{remark} There is an $\cS_4$-action on $S$: one can define
\[
S = \{(P, Q, R, T):  P+Q+R+T \in |K_D|\} \subset D \times D\times D \times D
\]
The action is obvious. $\pi : S \to D$ is always a fibration and there is an $\cS_3$-action on the fibers, which are the $C_P$.
\end{remark}

We compute now the numerical invariants of $S$.

\begin{theorem} For a generic $D$ the invariants of the surface $S$ are:
\label{invariants}
\begin{enumerate}
\item $q_S = 9$,
\item $c_2(S) = 96$,
\item $K_S^2 = 216$,
\item $p_g = 34$.
\end{enumerate}
\end{theorem}

\begin{proof}
We assume that all flexes of $D$ are simple, i.e., there are no points $Q\in D$ such that $|4Q| = K_D$. Under this hypothesis, the surface~$S$ is smooth. (cf. Remark~$\ref{r:S_smooth}$).

By Lemma~$\ref{qrel}$ we have $q_\pi = 6$  and so $q_S = q_\pi + g(D) = 9$.

We have seen in the previous proof that the fibration $\pi: S \to D$ has one singular fiber for each flex of~$D$, so it has $24$~singular fibers. Then
\[
c_2(S) = \chi_{top}(S) = \chi_{top}(D) \cdot \chi_{top}(C_P) + 24 = 96.
\] 

To compute $K_S^2$, we study the map $\beta : S \to D \times D$. Let $B \subset D\times D$ be the branch locus, so that
\[
B = \{(Q, R) \in D\times D \mid \text{$\overline{QR}$ is tangent to $D$}\}
\] 
and let $R\subset S$ be the ramification locus. Then
\[
K_S = \beta^*(K_{D\times D}) + R
\]
where $R$ is such that
\[
\beta_*R = B
\]
We fix some notation: if $P\in D$ is a point, we let as before $D_1 = \{P\} \times D$ and $D_2 = D \times \{P\}$ as numerical classes. Then 
\[
K_{D\times D} = 4D_1 + 4D_2
\]
and so to compute $K_S^2$ it is enough, by the projection formula, to compute the numerical class of~$B$. 

Recall the notation of Remark~$\ref{sym2}$: $\varphi : D\times D \to D^{(2)}$ is the quotient map to the symmetric product, $\Delta = \{(Q, Q) \mid Q \in D\} \subseteq D^{(2)}$ the diagonal, $\tau : D^{(2)} \to D^{(2)}$ the canonical involution given by $\tau(P + Q) = R + T$ where $P + Q + R + T$ is a canonical divisor on~$D$ and
\[
D_P = \{ P + Q \mid Q \in D\} \subseteq D^{(2)}
\]
($P$ is fixed). The class of $\Delta$ is divisible by~$2$ and we let $\delta = \frac12\Delta$. 

As $\varphi$ has degree~$2$, from Remark~$\ref{sym2}$ we obtain that $\varphi_*X_P = 2\tau_*D_P$ and from the proof of Theorem~$\ref{explicit}$ we know that the numerical class of~$X_P$ in~$D\times D$ is $3(D_1 + D_2) -\Delta_{D\times D}$. Moreover $\varphi_*D_1 = \varphi_*D_2 = D_P$ and $\varphi^*\delta = \Delta_{D\times D}$.

Hence 
\[
\tau_*D_P =  \dfrac12 \varphi_*X_P = \dfrac12(6D_P - \Delta) = 3D_P - \delta 
\]
and
\[
K_{D^{(2)}} = 4D_P - \delta
\]
The basic intersection numbers are
\[
D_P^2 = 1, \qquad D_P \cdot \delta = 1, \qquad \delta^2 = 1 - g_D = -2
\]
Since $\tau$ is an automorphism of $D^{(2)}$, we have $\tau_*K_{D^{(2)}} = K_{D^{(2)}}$ and so the canonical class $K_{D^{(2)}}$ is invariant under~$\tau$. From this we obtain $\tau_*\delta = 8D_P - 3\delta$ and hence $\tau_*\Delta = 16D_P - 6\delta$.

Looking at the composition
\[
\xymatrix{
S  \ar[r]^\beta  & D\times D \ar[r]^\varphi & D^{(2)}
}
\]
we have $\varphi^*(\delta) = \Delta_{D\times D}$ and observe that $\varphi^*(\tau(\Delta)) = B$, the branch locus of~$\beta$. In fact, if $(P, P)\in \Delta$, then $\tau(P,P) = Q + R$ and the line $\overline{QR}$ is tangent to $D$ and hence $\tau(P,P) \in B$. We finally obtain
\[
B = 16(D_1 + D_2) - 6\Delta_{D\times D}
\]
and we note that from the genus formula on $D\times D$ we have $\Delta_{D\times D}^2 = -4$.

We now show that $B$ is smooth. From the numerical class we can compute the arithmetic genus:
\[
p_a(B) = 1 + \dfrac12\,(B^2 + B\cdot K) = 33.
\]
On the other hand, the map $B \to D$ sending the point $(Q, R)$ to $P\in D$ where $Q + R + 2P$ is a canonical divisor of~$D$ is a double covering and since all flexes are simple it is ramified at the $56$~points $(Q, Q)$ where the tangent line is a bitangent.
The Riemann-Hurwitz formula then gives $g(B) = 33$ and so the geometric genus is equal to the arithmetic genus and hence $B$ is smooth. This shows again that $S$ is smooth.

We can then use the formula for the invariants of double coverings on page~$237$ of~\cite{BHPVdV}:
\[
K^2_S = 2 K_{D\times D}^2 + 4 L\cdot K_{D\times D} + 2 L \cdot L
\]
where $L = \dfrac12 B = 8(D_1 + D_2) - 3\Delta_{D\times D}$
to obtain
\[
K_S^2 = 216.
\]

From Noether's formula we also get $\chi(\cO_S) = 26$ and hence $p_g = 34$.
\end{proof} 

\begin{remark}
In the formulas given in~\cite{BHPVdV} there is also one for~$c_2(S)$, expressed in terms of the intersection product and $c_2(D\times D)$. Our computation is different since it uses the structure of~$S$ as a fibration. 
\end{remark}

\section{The case $g = 3$, $p = 3$}
\label{g3p3}
In this case, the map $\psi : \mathcal{H}_{3,3} \to \mathcal{M}_2$ has fibers $F = \psi^{-1}([D])$ of dimension~$2$. 
Using the construction of section~$\ref{Xiaofibr}$, a curve~$X \subset F$ gives a fibration $\pi : S \to B$.
For a general~$X$, the fibrations does not contradict Xiao's conjecture since $g_C = 7$ and the relative irregularity is $2g_D = 4$. We then look for special covers $D \to \bP^1$ so that the Galois closure~$C$ has geometric genus~$6$.

Let $D$ be a curve of genus~$2$, $P\in D$ not a Weierstrass point and let $f_P : D \to \bP^1$ be the map given by the linear series $|3P|$. Note that this $g^1_3$ is base point free since $P$ is not a Weierstrass point. We now do a construction similar to the previous case. Define the curve~$C_P$ as the closure of 
\[
\{(Q,R)\in D\times D \mid Q\neq R, f_P(Q)=f_P(R)\}
\]
and the induced map $\rho : C_P \to D$  of degree~$2$ is given by $\rho(Q,R) = T$ where $|Q + R + T| = |3P|$.

As in the proof of Theorem~$\ref{explicit}$, we can show that $C_P$ has a simple node at the point $(P,P) \in D \times D$. Choose a local coordinate~$z$ on $D$ centered at~$P$ and a local coordinate~$w$ on $\bP^1$ centered at~$f_P(P)$. In these coordinates the map is given locally by~$w = z^3$ and using the local coordinates on~$D\times D$ centered at~$(P, P)$ induced by~$z$, the points on the curve~$C_P$ different than $(P, P)$ are the pairs~$(x, y)$ such that~$x^3 = y^3$ and $x \ne y$. Then a local equation for~$C_P$ is $x^2 + xy + y^2 = 0$, which has a simple node at the origin. 

The curve $C_P$ is smooth in all other points~$Q$ unless $|3P| = |3Q|$. Since the $3$-torsion points in $J(D)$ are finite, for $P$ generic there are no such points~$Q$.

In this way we have a family $S_1$ parametrized by $D$ itself. We can describe this family explicitely in a way similar to the previous case: let $S_1 \subset D \times D \times D$ defined as
\[
S_1 = \{(P,Q,R): \exists T\in D: |3P| = |Q + R + T|\}.
\]
The projection on the first factor $\pi_1 : S_1 \to D$ has fibers the curves $C_P$ described above when $P$ is not a Weiertrass point and the map has a section $s : D \to S_1$ given by $s(P) = (P, P, P)$.

All the fibers of the fibration $\pi : S_1 \to D$ are singular and desingularizing along the section we obtain a new fibration
$\pi: S \to D$   with general smooth fiber of genus~$6$ and relative irregularity (at least)~$4$ and so we get a Xiao fibration. We note that the numbers are the same as in the case of $p = 5$, $g = 2$.

This is case~$3$ of Theorem~$\ref{maintheorem}$, which is now completely proved.

\end{document}